\documentclass[12pt]{article}
\usepackage{graphicx}
\usepackage{amssymb}
\usepackage{amsmath}
\usepackage{amsfonts}
\usepackage{amsthm}
\usepackage[english]{babel}
\usepackage[latin1,ansinew]{inputenc}
\usepackage{a4wide}
\usepackage{color}
\newcommand{\be}{\begin{eqnarray}}
\newcommand{\ee}{\end{eqnarray}}
\newtheorem{theo}{Theorem}

\newtheorem{lemma}{Lemma}

\newcommand{\R}{\mathbb R}

\begin{document}

\title{On Shock Profiles in Four-Field Formulations\\ of Dissipative Relativistic
Fluid Dynamics}

\author{\it Heinrich Freist\"uhler$^{1,a}$
{\rm and}\  Blake Temple$^{2,b}$}
\date{March 17, 2023}
\maketitle

Affiliations:\\
{\small
 1) Department of Mathematics and Statistics, University of Konstanz,
 78457 Konstanz, Germany\\
2) Department of Mathematics,
University of California at Davis, Davis CA 95616, USA

a) E-mail: {\tt heinrich.freistuehler@uni-konstanz.de}\\
b) E-mail: {\tt temple@math.ucdavis.edu} 
}

\vskip -4cm
\begin{abstract}
This paper shows that in second-order hyperbolic
systems of partial differential equations proposed in 
authors' earlier paper 
(J. Math.\ Phys.\ 59 (2018)) for modelling the relativistic dynamics of barotropic
fluids in the presence of viscosity and heat conduction,
shock waves of arbitrary strength have smooth, monotone dissipation profiles.
The results and arguments extend classical considerations of Weyl 
(Comm.\ Pure Appl.\ Math. {\bf 2} (1949))
and Gilbarg (Amer.\ J. Math.\ {\bf 73} (1951)) to the relativistic setting.
\end{abstract}

\newpage
\section{Introduction}\label{Introduction}
In the theory of relativity, the state of a barotropic fluid can be 
described by a 4-vector 
$
\psi^\alpha,\alpha=0,1,2,3,
$
which, as a function of the space-time coordinates $x^\beta,\beta=0,1,2,3,$ is governed
by a system of partial differential equations, 
\be\label{Teq}
{\partial \over \partial x^\beta}\left(T^{\alpha\beta}
+\Delta T^{\alpha\beta}
\right)
=0,\quad\alpha=0,1,2,3,
\ee
where in case viscosity and / or heat conduction are active,
\be\label{TintermsofB}
\Delta T^{\alpha\beta}=-B^{\alpha\beta\gamma\delta}\frac{\partial \psi_\gamma}{\partial x\delta}.
\ee
The tensors 
\be
T^{ \alpha\beta}\text{ and }B^{\alpha\beta\gamma\delta},\quad\alpha,\beta,\gamma,\delta=0,1,2,3,
\ee
are given functions of the four fields, i.e., the components of $\psi^\alpha$. 
%
In the absence of viscosity and heat 
conduction, the equations of motion reduce to the relativistic Euler equations 
\be\label{Euler}
{\partial \over \partial x^\beta}T^{\alpha\beta}
=0,
\quad \alpha=0,1,2,3.
\ee
The present paper focusses on shock waves, whose ideal version is given by discontinuous  
solutions to the latter, \eqref{Euler}, 
of the (prototypical) form 
\be\label{lsw}
\psi_\alpha(x)=\begin{cases}
       \psi_\alpha^-,&x^\beta \xi_\beta<0,\\
       \psi_\alpha^+,&x^\beta \xi_\beta>0,
       \end{cases}
\ee
and asks whether they can be properly represented in the dissipative setting. 
A standard way to achieve such representation is a `dissipation profile', i.e., 
a regular solution of \eqref{Teq} that depends also only on $x^\beta \xi_\beta$
and connects the two states forming the shock, in other words, a solution $\hat\psi$
of the ODE  
\be\label{profeq}
\xi_\beta\xi_\delta B^{\alpha\beta\gamma\delta}(\hat\psi)\hat\psi_\gamma'
=
\xi_\beta T^{\alpha\beta}(\hat\psi)-q^\alpha,
\quad 
q^\alpha:=\xi_\beta T^{\alpha\beta}(\psi\pm),
\ee
on $\R$ which is heteroclinic to them,
\be
\label{hetero}
\hat\psi_\alpha(-\infty)=\psi_\alpha^-,\quad  
\hat\psi_\alpha(+\infty)=\psi_\alpha^+.
\ee
Concretely, the state variables of a barotropic fluid are given by
$$
\psi^\alpha=\frac{U^\alpha}\theta,
$$
where $U^\alpha$ and $\theta$ are the 4-velocity and the temperature, 
$$
\theta=\left(-\psi_\alpha\psi^\alpha\right)^{-1/2}, 
$$
the fluid is specified  
by prescribing its pressure as a function of the temperature,
$$
p=\tilde p(\theta),
$$  
and the ideal part of the energy-momentum tensor is given by  
\be\label{idealgas}
T^{\alpha\beta}=\frac{\partial (\tilde p(\theta)\psi^\beta)}{\partial\psi_\alpha}
=\theta^3\frac{d\tilde p(\theta)}{d \theta}\psi^\alpha\psi^\beta
+\tilde p(\theta)g^{\alpha\beta}.
\ee
We assume strict causality 
in the sense that
\be\label{Hessianpositive}
\left(\frac{\partial^2(\tilde p(\theta)\psi^\beta)}{\partial \psi_\alpha\partial\psi_\gamma}T_\beta\right)_{\alpha,\gamma=0,1,2,3}
\text{is negative definite, all future non-spacelike directions }T_\beta. 
\ee
As regards the dissipative part 
we specify, following \cite{FT1,FT2,FT3}, 
$-\Delta T^{\alpha\beta}$ as\footnote{We use the metric $g^{\alpha\beta}=(-+++)$ and the projector 
$\Pi^{\alpha\beta}=g^{\alpha\beta}+U^\alpha U^\beta$.}
\be\label{DeltaT}
\begin{aligned}
-\Delta T^{\alpha\beta}_\Box
&=&\eta \Pi^{\alpha\gamma}\Pi^{\beta\delta}
\left[\frac{\partial U_{\gamma}}{\partial x^{\delta}}
+\frac{\partial U_{\delta}}{\partial x^{\gamma}}
-\frac{2}{3}g_{\gamma\delta}\frac{\partial U^{\epsilon}}{\partial x^{\epsilon}}\right]
+\check\zeta \Pi^{\alpha\beta}\frac{\partial U^{\gamma}}{\partial x^{\gamma}}\\
&&+\sigma\left[U^{\alpha}U^{\beta}\frac{\partial U^{\gamma}}{\partial x^{\gamma}}
-\left(\Pi^{\alpha\gamma}U^{\beta}
+\Pi^{\beta\gamma}U^{\alpha}\right)
U^{\delta}\frac{\partial U_{\gamma}}{\partial x^{\delta}}\right]
\\
&&+
\chi
\bigg[
\bigg(U^\alpha\frac{\partial \theta}{\partial x_\beta}
+U^\beta\frac{\partial \theta}{\partial x_\alpha}\bigg)
-g^{\alpha\beta}U^\gamma\frac{\partial \theta}{\partial x^\gamma}
\bigg],
\end{aligned}
\ee
where 
\be\label{choicetildezeta}
\sigma=((4/3)\eta+\zeta)/(1-c_s^2)-c_s^2\chi\theta
\quad\text{and}\quad
\check\zeta=\zeta+c_s^2\sigma-c_s^2(1-c_s^2)\chi\theta
\ee
with $\eta,\zeta,\chi$ the coefficients of shear viscosity, bulk viscosity, and thermal conductivity,
and $0<c_s<1$ the speed of sound.

The following are our main results.

\begin{theo}\label{ThmViscousProfilesBaro1}
Consider a barotropic fluid with viscosity and without heat conduction. 
Assume that the acoustic mode is genuinely nonlinear. Then any Lax shock 
has a dissipation profile with respect to $\Delta T_\Box$.
\end{theo}

\begin{theo}\label{ThmViscousProfilesBaro2}
Consider a barotropic fluid with viscosity and with heat conduction. 
Assume that the acoustic mode is genuinely nonlinear. Then any Lax shock 
has a dissipation profile with respect to $\Delta T_\Box$.
\end{theo}

Sections 2 and 3 are devoted to the proofs of Theorems 1 and 2, respectively.
In Section 4, we 
contrast these findings with recently established properties 
of a different formulation that was proposed by Bemfica, Disconzi, and
Noronha in \cite{BDN}.

\section{Shock profiles in viscous barotropic fluids {without heat conduction}}
\setcounter{equation}0
To demonstrate Theorem 1, 
we write the pressure also as a function
$$
p=\hat p(\rho)
$$
of the energy 
$$
\rho=\theta \tilde p'(\theta)-\tilde p(\theta), 
$$
to which the sound speed is of course connected as $c_s^2=\hat p'(\rho)$,   
and recall from \cite{CB69,Bo,ST} that 
genuine nonlinearity of the acoustic mode
is characterized by the condition
\be\label{GNL}
(\rho+\hat p(\rho))\hat p''(\rho)+2(1-\hat p'(\rho))\hat p'(\rho)>0.
\ee
Consider an ideal shock wave \eqref{lsw}
and assume w.\ l.\ o.\ g.\ that the spatiotemporal direction of 
propagation is $(0,1,0,0)$, i.e., $\xi_\beta=\delta_{\beta1}$, and 
$\psi_2=\psi_3=0$ (as can always be achieved by a Lorentz transformation).
The profile ODE system, 
\be\label{profileode}
-(\Delta T^{\alpha1}_\Box)'
=
T^{\alpha1}-q^\alpha,
\ee
then has two active equations, $\alpha=0,1$. 

We first characterize the rest points in terms of their dependence on the free constant $q^\alpha$.

\begin{lemma}
For every $q^1>0$ there exists a unique $Q(q^1)>0$ such that the following holds:
The algebraic system $$T^{\alpha1}(\cdot)=q^\alpha,\quad\alpha=0,1,$$ has 
more than one solution if and only if 
\be\label{qqQ}
q_1^2<q_0^2<q_1^2+Q(q^1),
\ee
in which case it has precisely two solutions.
\end{lemma}
\begin{proof}
To see this, note first that the equation $T^{11}(.)=q^1$ is equivalent to
$$
p<q^1\quad \text{and}\quad 
u_1^2=\frac{q^1-p}{\rho+p}.
$$
Under this condition, the equation $T^{01}(.)=q^0$ is equivalent to the combination 
of 
\be\label{q1}
(\rho+p)^2\frac{q^1-p}{\rho+p}\left(1+\frac{q^1-p}{\rho+p}\right)=q_0^2
\ee
and
\be\label{guterichtung}
u^1q^0>0.
\ee
We write \eqref{q1} as
\be\label{geq}
g(\rho)\equiv -\rho\hat p(\rho)+q^1(-\hat p(\rho)+\rho)=q_0^2-q_1^2
\ee
with $g$ defined on the interval $$I\equiv [0,\bar\rho] \quad \text{with }\hat p(\bar\rho)=q^1.$$ 
Note now that any stationary point of $g$ is a nondegenerate maximum. 
This follows as assuming 
$0=g'(\rho)$
at some point $\rho \in I$ implies 
$$
q^1=\frac{\hat p(\rho)+\rho\hat p'(\rho)}{1-\hat p'(\rho)}
$$
and thus, using \eqref{GNL}, 
\be\label{gprimeprime}
g''(\rho)=-(\rho+q_1)p''(\rho)-2\hat p'(\rho)
=-\frac{\rho+\hat p(\rho)}{1-\hat p'(\rho)}\hat p''(\rho)-2\hat p'(\rho)
<0.
\ee
As $g'(0)>0$, this means that along $I$, $g$ increases 
from $g(0)=0$ to $Q\equiv\max_I g>0$ and then decays to $g(\bar\rho)=-q_1^2<0$,
the monotonicity being strict in both parts. Equation \eqref{geq} thus has more than one, namely 
two, positive solutions if and only \eqref{qqQ} holds. 
\end{proof}
Returning to our shock wave \eqref{lsw}, we note that it must correspond to certain parameter 
values $q^0,q^1$ with the properties recorded in Lemma 1; regarding \eqref{guterichtung}, 
we fix signs\footnote{The opposite case $u^1,q^0<0$ differs only by a transformation 
$x^1\mapsto -x^1$.}  
as
$$
u^1,q^0>0. 
$$
Since
$$
-(\Delta T^{\alpha1}_\Box)'=\sigma (u^\alpha)',
$$
we can rewrite the traveling wave system \eqref{profileode} equivalently as
\begin{eqnarray}
\sigma u_\alpha (u^\alpha)'
&=&
u_\alpha (T^{\alpha1}-q^\alpha)
\label{profilepassive}
\\
\sigma v_\alpha(u^\alpha)'
&=&
v_\alpha (T^{\alpha1}-q^\alpha)
\label{profileactive}
\end{eqnarray}
with $(v^0,v^1)=(u^1,u^0)$ orthogonal to $(u^0,u^1)$. 
As $u_\alpha (u^\alpha)'=0$, equation 
\eqref{profilepassive} is an algebraic constraint,
\be\label{algcond}
q^0u^0-(\rho+q^1)u^1=0,
\ee
or
\be
u_1=U(\rho):=\left(\left(\frac{\rho+q^1}{q^0}\right)^2-1\right)^{-1/2}
\ee
By virtue of \eqref{algcond}, 
\be
u^0v_\alpha(u^\alpha)'=
u^0(v^0 u_0'+v^1 u_1')=u_1',
\ee
and
\be
u^0v_\alpha(T^{\alpha1}-q^\alpha)=u^0v_\alpha((\rho+p)u^\alpha u^1+pg^{\alpha1 }-q^\alpha)=
u^0(-v_0q^0+v_1(p-q^1)),
\ee
equation \eqref{profileactive} reduces to
$$
\sigma u_1'=(\rho+p)u_1^2+(p-q^1)
$$
or, using \eqref{algcond} again,
$$
\sigma U'(\rho)\rho'= 
R(\rho):=
(\rho+p)\left(\left(\frac{\rho+q^1}{q^0}\right)^2-1\right)^{-1}+(p-q^1)=
(\rho+p)\left(\frac{\rho+q^1}{q^0}\right)^2-(\rho+q^1).
$$
As $R>0$ between its two zeros, the heteroclinic solution connects them, and $\rho$ 
increases in the direction in which the fluid moves. 
The latter correctly fits the fact that Lax shocks are compressive.  

\emph{Remarks.} 
(i) The argument, notably inequality \eqref{gprimeprime}, reveals the geometric meaning of the genuine 
nonlinearity condition \eqref{GNL} for the Rankine-Hugoniot relations. 
\\
(ii) For the special case of pure radiation, $p=\rho/3$, the result updates 
considerations in \cite{FT1} to the dissipation tensor $\Delta T_\Box$ as derived 
in \cite{FT3}, which
(cf.\ a remark towards the end of Sec.\ 1 in \cite{FT3}) is not exactly identical with the
version originally proposed in \cite{FT1}.
 
\section{Shock profiles in viscous barotropic fluids with heat conduction}
\setcounter{equation}0
To include both viscosity and heat conduction, we have to use the full dissipation tensor   
\eqref{DeltaT}. 
In this case, we can work directly with \eqref{profeq}; in other words, we express 
\eqref{profileode} via \eqref{TintermsofB} as 
  \be\label{profeq1d}
B^{\alpha1 \gamma1}(\hat\psi)\hat\psi_\gamma'
=
T^{\alpha1}(\hat\psi)-q^\alpha,
\quad 
q^\alpha:=T^{\alpha1}(\psi^\pm)
\ee
and consider this system on its natural domain of definition,
$$
\Psi=\{(\psi^0,\psi^1)\in\R^2:\psi^0>|\psi^1|\},
$$
with $q^\alpha$ corresponding to a fixed Lax shock $\psi^-\to\psi^+$ as in Sec.\ 2,
the characteristic speeds $\lambda_{1,2}$ being both positive at $\psi^-$ and 
having different signs at $\psi^+$. 

Since
$$ 
T^{\alpha\beta}=\theta^3p'(\theta)\psi^\alpha\psi^\beta
+p(\theta)g^{\alpha\beta}=\frac{\partial L^\beta(\psi)}{\partial\psi_\alpha}
$$ 
with
$$
L^\beta(\psi)\equiv p(-(\psi^ \gamma\psi_\gamma)^{-1/2})\psi^\beta,
$$
and due to sharp causality $B^{\alpha 1\gamma 1}$ is positive definite, and $L$ with
$$L(\psi)=
L^1(\psi)-q^\gamma\psi_\gamma
$$ 
is a strict Liapunov function for \eqref{profeq1d}.

The Jacobian of the vector field $F^\alpha=T^{\alpha 1}-q^\alpha$ is the Hessian of $L$,
$$
\frac{\partial F^\alpha}{\partial \psi_\beta}
=
\frac{\partial^2L^1(\psi)}{\partial\psi_\alpha\partial\psi_\beta}
=
\frac{\partial^2L(\psi)}{\partial\psi_\alpha\partial\psi_\beta}
=:H^{\alpha\beta}(\psi),
$$
and the eigenvalues of $H^{\alpha\beta}(\psi)$
relative to the positive definite matrix 
$$
\frac{\partial T^{\alpha 0}}{\partial \psi_\beta}
=
\frac{\partial^2L^0(\psi)}{\partial\psi_\alpha\partial\psi_\beta}>0
$$
are the characteristic speeds $\lambda_{1,2}$ at the fluid state $\psi$. 
If we choose $q^\alpha$ and signs as in Sec.\ 2, these speeds
are both positive at the left hand state $\psi^-$, and of different signs at the right hand 
state $\psi^+$ of the shock. This implies that 
$$
H^{\alpha\beta}(\psi^-)>0\quad\text{and}\quad\det H^{\alpha\beta}(\psi^+)<0;
$$
the exactly two critical points of the Liapunov function $L$ thus are 
a strict local minimum at $\psi^-$ and a hyperbolic saddle at $\psi^+$. 

For a shock of sufficiently small amplitude, i.e., for a parameter value 
$q=(q^0,q^1)$ with $q^1>0$ while $q_0^2-q_1^2>0$ is small enough, the shock profile 
exists, i.e., system \eqref{profeq1d} possesses a heteroclinic orbit 
with $\alpha$-limit $\psi^-$ and $\omega$-limit $\psi^+$. Consequently,
for this value of $q$ 
$$
c^+\equiv L(\psi^+)>L(\psi^-)\equiv c^-.
$$
and the $c^+$ level line of $L$ contains a closed curve (with a corner at $\psi^+$) 
whose interior 
$$
\Omega\subset L^{-1}((-\infty,c^+))\cap W^u(\psi^-) 
$$
contains $\psi^-$.
%
%
Now, on the one hand, as
$$
\frac{\partial L}{\partial\psi_1}(\psi^0,\psi^1)
=
p(\theta)+\theta^3p'(\theta)\psi_1^2-q^1 >0
$$
for sufficiently small 
$$
\theta^{-2}=-\psi^ \gamma\psi_\gamma<\psi_0^2,
$$
$L$ increases strictly on the line segment 
$$
S \equiv \{(\psi^0,\psi^1)\in\Psi:\psi^0 =\sigma\} 
$$   
if $\sigma>0$ is chosen small enough (in dependence on $q^1$). 
On the other hand, $L$ tends to $\pm \infty$ near $\partial\Psi$,
$$
\lim_{\psi^1\searrow-\psi^0}  L(\psi^0,\psi^1)=-\infty,\quad 
\lim_{\psi^1\nearrow+\psi^0}  L(\psi^0,\psi^1)=+\infty.
$$
Therefore the situation, including the shock profile 
is robust 
against perturbations of $q$ within the range given by \eqref{qqQ}. This means 
that every Lax shock has a profile. 

\section{Non-existence of shock profiles in the Bemfica-Disconzi-Noronha model}
\setcounter{equation}0
In 1940, Eckart has proposed equations for dissipative relativistic fluid dynamics \cite{E}. 
For barotropic fluids, the Eckart theory 
reads
\be \label{TeqE} 
{\partial \over \partial x^\beta}\left(T^{\alpha\beta}
+\Delta T^{\alpha\beta}_{E}
\right)
=0,
\ee
with
%
%
\begin{eqnarray}
 \label{DeltaTEL}
-\Delta T^{\alpha\beta}_{E}=
\eta \Pi^{\alpha\gamma}\Pi^{\beta\delta}\left[\frac{\partial U_{\gamma}}{\partial x^{\delta}}+\frac{\partial U_{\delta}}{\partial x^{\gamma}}-\frac{2}{3}g_{\gamma\delta}\frac{\partial U^{\epsilon}}{\partial x^{\epsilon}}\right]
+\zeta \Pi^{\alpha\beta}\frac{\partial U^{\gamma}}{\partial x^{\gamma}}
+
\chi\left(\Pi^{\alpha\gamma}U^{\beta}+\Pi^{\beta\gamma}U^{\alpha}\right)
\frac{\partial\theta}{\partial x^{\gamma}}
\end{eqnarray}
where, as before, $\eta,\zeta,\chi>0$ denote the fluid's coefficients of shear and bulk viscosity, and heat conduction. 
Sometimes referred to as `relativistic Navier-Stokes', the Eckart equations \eqref{TeqE} 
have been fundamental to the development of fluid dynamics. 
However, like the classical Navier-Stokes equations, they do not have 
the property of finite speed of propagation. This violates causality, i.e., the principle that  
speeds of propagation must not be larger than that of light.

Since physically, Navier-Stokes is a first order theory, in \cite{FT3} we address the causality issue 
by introducing a rigorous notion of when two different dissipation tensors $\Delta T^{\alpha\beta}$ produce equations
which are equivalent in a first order sense, and    
show the following.
\begin{theo}\label{thmCE} 
(i) Our hyperbolic Navier-Stokes system \eqref{Teq}
with $\Delta T^{\alpha\beta}=\Delta T^{\alpha\beta}_\Box$  and 
\eqref{choicetildezeta}
is causal in the sense that all signal speeds are not larger than that of light. 
(ii) The system is  
first-order equivalent to the Eckart system \eqref{TeqE}.
\end{theo}
Here, first-order equivalence is defined in terms of gradient-expansion transformations between different four-field theories 
\eqref{Teq} of dissipative fluid dynamics -- two systems are equi\-valent if the  change in $\Delta T$ occurs only at second 
order in the magnitude of the 
dissipation coefficients. See \cite{FT3,FT2} for a precise definition and a formal algebraic setup of first-order equi\-valence 
transformations.  

Our theory with $\Delta T^{\alpha\beta}=\Delta T^{\alpha\beta}_\Box$ is not the only formulation \eqref{Teq} that 
is first-oder equivalent to Eckart's. After all, the prominent theory given in Landau and Lifshitz \cite{LL} also is,
while it is again not causal. 
In \cite{BDN}, Bemfica, Disconzi, and Noronha have proposed for the dynamics of the pure radiation fluid, $p(\theta)=\theta^4/3$,
another four-field PDE formulation that is first-order equivalent to Eckart's (with $\zeta=\chi=0$), namely    
\be\label{NS}
\partial_\beta(T^{\alpha\beta}+\Delta T^{\alpha\beta}_{BDN})=0
\ee
by setting 
\be
\label{DeltaTBDN}
-\Delta T^{\alpha\beta}_{BDN}=
B^{\alpha\beta\gamma\delta}_{BDN}
\displaystyle{\frac{\partial\psi_\gamma}{\partial x^\delta}},
\ee
with 
\be\label{Bis}
B^{\alpha\beta\gamma\delta}_{BDN}
=\eta
B^{\alpha\beta\gamma\delta}_E-
\mu B^{\alpha\beta\gamma\delta}_1-\nu B^{\alpha\beta\gamma\delta}_2
\ee
where the classical Eckart viscosity tensor 
\be
\begin{aligned}\label{Bvisc}
B^{\alpha\beta\gamma\delta}_E&=
\Pi^{\alpha\gamma}\Pi^{\beta\delta}+\Pi^{\alpha\delta}\Pi^{\beta\gamma}
-\frac23\Pi^{\alpha\beta}\Pi^{\gamma\delta}
\end{aligned}
\ee
is augmented via
\be\label{Bthervelo}
B^{\alpha\beta\gamma\delta}_1=
(3 U^\alpha U^\beta+\Pi^{\alpha\beta})
(3 U^\gamma U^\delta+\Pi^{\gamma\delta}),
\quad
B^{\alpha\beta\gamma\delta}_2=
(U^\alpha {\Pi^\beta}_\epsilon+ U^\beta{\Pi^\alpha}_\epsilon) 
(U^\gamma\Pi^{\delta\epsilon}+ U^\delta\Pi^{\gamma\epsilon}).
\ee
It was shown in \cite{BDN} that 
this formulation, which we will here briefly refer to as `the BDN model', is indeed causal if and only if, 
relative to the classical coefficient $\eta$ of viscosity, the
coefficients $\mu$ and $\nu$ of the ``regulators'' 
$ 
B^{\alpha\beta\gamma\delta}_1,
B^{\alpha\beta\gamma\delta}_2
$ 
satisfy
\be\label{generallycausal}
\mu\ge \frac43\eta\quad\text{and}\quad 
\nu\le\left(\frac1{3\eta}-\frac1{9\mu}\right)^{-1}.
\ee
On the other hand the following theorem was proven in \cite{F21}: 
\begin{theo}
If the dissipation coefficients $\eta,\mu,\nu>0$ satisfy the strict causality condition
\be\label{strictlycausal}
\mu\ge \frac43\eta\quad\text{and}\quad 
\nu< \left(\frac1{3\eta}-\frac1{9\mu}\right)^{-1},
\ee
then the BDN model always possesses Lax shocks that do not admit any dissipation profile.
\end{theo}
\vskip -.3cm
This contrasts sharply with our Theorems 1 and 2 above. 

Theorem 4 would not preclude the possibility that for some \emph{sharply causal} choice of $\mu$ and $\nu$, i.e., 
in the case 
\be\label{sharplylycausal}
\mu\ge \frac43\eta\quad\text{and}\quad 
\nu=\left(\frac1{3\eta}-\frac1{9\mu}\right)^{-1}
\ee
(cf.\ \cite{F21}), all Lax shocks do have profiles again. However, Pellhammer has shown \cite{P}:  
\begin{theo}
Whatever values $\eta,\mu,\nu>0$ with \eqref{sharplylycausal} are assumed, the BDN model always possesses a range of Lax 
shocks that have either no profile at all or an oscillatory profile, i.e., a profile whose orbit infinitely spirals around one 
of the endstates.        
\end{theo}
Compared with non-relativistic gas dynamics \cite{Gi}, this property seems exotic. It would be interesting to know 
whether oscillatory shock profiles in the BDN model are dynamically stable.  
 
For positive stability results on shock profiles in hyperbolically regularized systems of conservation laws see \cite{B}.  

{\bf Conclusion.} Our formulation \eqref{Teq}, \eqref{DeltaT} of relativistic Navier-Stokes is causal, 
takes a solvable hyperbolic form, is physically justified by its leading-order equivalence with Eckart and Landau-Lifshitz, 
and, like classical Navier-Stokes, appears phenomenologically correct notably in the sense that it captures all shock waves
consistently.

\end{document}